\documentclass[english]{article}
\usepackage[T1]{fontenc}
\usepackage[latin9]{inputenc}
\usepackage{amsthm}
\usepackage{amsmath}
\usepackage{amssymb}
\usepackage{esint}

\makeatletter
\numberwithin{equation}{section}
  \theoremstyle{definition}
  \newtheorem{defn}{\protect\definitionname}[section]
  \theoremstyle{plain}
  \newtheorem{lem}{\protect\lemmaname}[section]
  \theoremstyle{remark}
  \newtheorem{rem}{\protect\remarkname}[section]

\makeatother

\usepackage{babel}
  \providecommand{\definitionname}{Definition}
  \providecommand{\lemmaname}{Lemma}
  \providecommand{\remarkname}{Remark}

\begin{document}

\title{Note on the Signatures of Rough Paths in a Banach Space }

\author{Horatio Boedihardjo%
\thanks{Department of Mathematics and Statistics, Reading University, England.
The main part of his contribution on this note was made while at Oxford-Man
Institute, Eagle House, Walton Well Road, Oxford OX2 6ED, England.
\protect \\
Email: h.s.boedihardjo@reading.ac.uk %
} , Xi Geng%
\thanks{Mathematical Institute, Woodstock Road, Oxford OX2 6GG and Oxford-Man
Institute, Eagle House, Walton Well Road, Oxford OX2 6ED, England.\protect \\
Email: xi.geng@maths.ox.ac.uk %
}, Terry Lyons%
\thanks{Mathematical Institute, Woodstock Road, Oxford OX2 6GG and Oxford-Man
Institute, Eagle House, Walton Well Road, Oxford OX2 6ED, England.\protect \\
Email: terry.lyons@oxford-man.ox.ac.uk %
} and Danyu Yang%
\thanks{Oxford-Man Institute, Eagle House, Walton Well Road, Oxford OX2 6ED,
England.\protect \\
Email: danyu.yang@oxford-man.ox.ac.uk %
}}
\maketitle
\begin{abstract}
We prove some results, which are used in \cite{our uniqueness paper },
about weakly geometric rough paths that are well-known in finite dimensions,
but need proof in the infinite dimensional setting. 
\end{abstract}

\section{Introduction }

Lyons' rough path theory \cite{MR1654527} investigates the meaning
of the controlled differential equation 
\begin{equation}
\mathrm{d}y_{t}=V\left(y_{t}\right)\mathrm{d}x_{t}\label{eq:controlled differential equation}
\end{equation}
when $x$ does not have a derivative. There are two natural classes
of paths $x$ for which the equation (\ref{eq:controlled differential equation})
is defined and the operation of integration satisfies the chain rule
of classical calculus: the \textit{geometric rough paths} and the
\textit{weakly geometric rough paths.} Let $p\geq1$. The $p$-geometric
rough paths (respectively $p$-weakly geometric rough paths) are the
geometric rough paths (respectively weakly geometric rough paths)
that have finite $p$-variation (see Definition \ref{weakly geometric rough paths}
below). The \textit{signature} of a rough path $x$ on the interval
$\left[0,T\right]$, defined formally as 
\[
S\left(x\right)=1+\int_{0}^{T}\mathrm{d}x_{t_{1}}+\int_{0}^{T}\int_{0}^{t_{2}}\mathrm{d}x_{t_{1}}\otimes\mathrm{d}x_{t_{2}}+\ldots,
\]
(see Definition \ref{signature definition} for precise definition)
which arises naturally in rough path theory. In finite dimensions,
a $p$- weakly geometric rough path is a $p^{\prime}$-geometric rough
path for all $p^{\prime}>p$ (see \cite{FV06}). As the map $x\rightarrow S\left(x\right)$
is continuous in the $p$-variation topology for all $p$, most results
known for the signatures of bounded variation path will automatically
extend to $p$-geometric rough paths for all $p$, and hence will
also hold for $p^{\prime}$-weakly geometric rough paths for all $p^{\prime}$.
This tool for proving properties of signatures of weakly geometric
rough paths break down in infinite dimensions. It is not known whether,
in infinite dimensions, there exists weakly geometric rough paths
that are not $p$-geometric rough paths for any $p$. As a result,
to complete the proofs for the main result in \cite{our uniqueness paper },
it becomes necessary to check that the certain properties of signatures
of weakly geometric rough paths still holds in infinite dimensions.
In the time between writing the first draft of this article and posting
this article, the article \cite{Cass infinite dimensional rp} has
appeared which came up with the idea of using a Lemma similar to Lemma
\ref{finite dimensional projection} to extend algebraic identity
in finite dimension to infinite dimensions. For the convenience of
the reader, we will leave the proof of Lemma \ref{finite dimensional projection}
and its application in this paper.

\subsection*{Notations and main results }

For vector spaces $A$ and $B$, we shall use $A\otimes_{a}B$ to
denote the algebraic tensor product of $A$ and $B$, namely, 
\[
A\otimes_{a}B=\mbox{span}\left\{ a\otimes b:a\in A,\; b\in B\right\} .
\]
Let $V$ be a Banach space. We assume that each $V^{\otimes_{a}k}$
is equipped with a norm $\left\Vert \cdot\right\Vert _{V^{\otimes_{a}k}}$
so that:

1. The family $\left\{ \left\Vert \cdot\right\Vert _{V^{\otimes_{a}k}}:k\geq0\right\} $
satisfies 
\begin{equation}
\|a\otimes b\|_{V^{\otimes_{a}(m+n)}}\leq\|a\|_{V^{\otimes_{a}m}}\cdot\|b\|_{V^{\otimes_{a}n}}\label{eq:admissible norm}
\end{equation}
 for $a\in V^{\otimes_{a}m}$ and $b\in V^{\otimes_{a}n}$. 

2. Let $n\in\mathbb{N}$. For any permutation $\sigma$ on $\left\{ 1,\ldots,n\right\} $
and $v_{1},\ldots,v_{n}\in V$,
\[
\left\Vert v_{1}\otimes\ldots\otimes v_{n}\right\Vert _{V^{\otimes_{a}n}}=\left\Vert v_{\sigma\left(1\right)}\otimes\ldots\otimes v_{\sigma\left(n\right)}\right\Vert _{V^{\otimes_{a}n}}.
\]
For $m,n\in\mathbb{N}$, if we define $V^{\otimes m}\otimes V^{\otimes n}$
to be the completion of $V^{\otimes m}\otimes_{a}V^{\otimes n}$ under
the norm $\|\cdot\|_{V^{\otimes_{a}(m+n)}},$ we have 
\[
V^{\otimes_{a}(m+n)}\subset V^{\otimes m}\otimes_{a}V^{\otimes n}\subset V^{\otimes(m+n)}.
\]
It follows that 
\[
V^{\otimes m}\otimes V^{\otimes n}\backsimeq V^{\otimes(m+n)},
\]
where $\backsimeq$ denotes the isomorphism as Banach spaces. Let
$\tilde{T}^{\left(n\right)}\left(V\right)$ denote the truncated formal
series of tensors whose scalar component is $1$, that is 
\[
\tilde{T}^{\left(n\right)}\left(V\right)=1\oplus V\oplus V^{\otimes2}\oplus\ldots\oplus V^{\otimes n}.
\]
Let $\tilde{T}((V))$ be the space of formal series of tensors, which
are sequences $\left(a_{j}\right)_{j=0}^{\infty}$with $a_{0}=1$
and $a_{j}\in V^{\otimes j}$ equipped with the addition and multiplication
\begin{eqnarray*}
\left(a+b\right)_{i} & = & a_{i}+b_{i};\\
\left(a\otimes b\right)_{i} & = & \sum_{i=0}^{k}a_{i}\otimes b_{k-i}.
\end{eqnarray*}
We will use $\pi^{\left(n\right)}$ and $\pi_{n}$ to denote the projection
maps from the formal series of tensors $\tilde{T}\left(\left(V\right)\right)$
onto the truncated tensors $\tilde{T}^{\left(n\right)}\left(V\right)$
and $V^{\otimes n}$ respectively. 

We will define the space $L_{a}^{n}\left(V\right)$ of Lie polynomials
over $V$ of degree $n$ inductively so that if $\left[x,y\right]=x\otimes y-y\otimes x$,
then
\begin{eqnarray*}
L_{a}^{1}\left(V\right) & = & V\\
L_{a}^{n+1}\left(V\right) & = & \mbox{Span}\left\{ \left[x,y\right]:x\in L_{a}^{n}\left(V\right),\; y\in V\right\} .
\end{eqnarray*}
We will let 
\begin{eqnarray*}
\mathcal{L}_{a}^{n}\left(V\right) & = & \bigoplus_{i=1}^{n}L_{a}^{i}\left(V\right)
\end{eqnarray*}
The exponential map on $\tilde{T}^{\left(n\right)}\left(V\right)$
is defined by 
\[
\exp\left[l\right]=\pi^{\left(n\right)}\left[\sum_{j=0}^{n}\frac{l^{\otimes j}}{j!}\right],
\]
with the convention that $l^{\otimes0}:=\mathbf{1}$. Define a function
$\left\Vert \cdot\right\Vert :\tilde{T}^{\left(n\right)}\left(V\right)\rightarrow\mathbb{R}$
by

\begin{eqnarray}
\|\cdot\|_{\tilde{T}^{n}\left(V\right)}: & = & \max_{1\leq k\leq n}\|\pi_{k}\left(\cdot\right)\|_{V^{\otimes k}}^{\frac{1}{k}}.\label{eq:norm for tensors}
\end{eqnarray}
We will let $G^{\left(n\right)}$ denote $\overline{\exp\mathcal{L}_{a}^{\lfloor p\rfloor}\left(V\right)}$. 
\begin{defn}
\label{weakly geometric rough paths}We say a path $x:\left[0,T\right]\rightarrow\tilde{T}^{\left(n\right)}\left(V\right)$
has finite $p$-variation if 
\[
\sup_{\mathcal{P}}\left(\sum_{t_{j}\in\mathcal{P}}\|x_{t_{j}}^{-1}x_{t_{j+1}}\|_{\tilde{T}^{\left(n\right)}\left(V\right)}^{p}\right)^{\frac{1}{p}}<\infty.
\]
Let $p\geq1$. We say a path $x:\left[0,T\right]\rightarrow\tilde{T}^{\left(\lfloor p\rfloor\right)}\left(V\right)$
is a $p$-weakly geometric rough path if $x_{t}\in G^{\left(\lfloor p\rfloor\right)}$
and $x$ has finite $p$-variation. 
\end{defn}
We are mainly interested in the following special transform of weakly
geometric rough paths.
\begin{defn}
\label{signature definition}(Lyons, Theorem 2.2.1 \cite{MR1654527})
For all $p$-weakly geometric rough paths $x:\left[0,T\right]\rightarrow\tilde{T}^{\left(\lfloor p\rfloor\right)}\left(V\right)$
and all $n\geq\lfloor p\rfloor$, there exists a unique $S_{n}\left(x\right)_{0,\cdot}:\left[0,T\right]\rightarrow\tilde{T}^{\left(n\right)}\left(V\right)$
such that $S_{n}\left(x\right)_{0,0}=\mathbf{1}$, $S_{n}$ has finite
$p$-variation and 
\[
\pi^{\left(\lfloor p\rfloor\right)}\left(S_{n}\left(x\right)_{0,t}\right)=x_{0}^{-1}x_{t}.
\]
 for all $t$. The unique element $S\left(x\right)_{0,T}\in\tilde{T}\left(\left(V\right)\right)$
such that $\pi^{\left(n\right)}\left(S\left(x\right)_{0,T}\right)=S_{n}\left(x\right)_{0,T}$
for all $n$ is called the signature of the path $x$. 
\end{defn}

\subsection*{Weakly geometric rough path in Banach space }

The results we will prove in this section are: 
\begin{lem}
\label{symmetric norm}The function $\left\Vert \cdot\right\Vert _{\tilde{T}^{\left(n\right)}\left(V\right)}$,
defined in (\ref{eq:norm for tensors}), is symmetric on $G^{\left(n\right)}$
in the sense that for all $g\in G^{\left(n\right)}$, 
\[
\left\Vert g\right\Vert _{\tilde{T}^{\left(n\right)}\left(V\right)}=\left\Vert g^{-1}\right\Vert _{\tilde{T}^{\left(n\right)}\left(V\right)}.
\]

\end{lem}

\begin{lem}
\label{key Lemma}Let $N\in\mathbb{N}$. There exists a map $\mathcal{J}:WG\Omega_{p}\left(V\right)\rightarrow WG\Omega_{p}\left(\bigoplus_{i=0}^{N}V^{\otimes i}\right)$
such that for all $x\in WG\Omega_{p}\left(V\right)$:

1. $\pi_{1}\left(\mathcal{J}\left(x\right)\right)=S_{N}\left(x\right)_{0,\cdot}$;

2. If $x,y\in WG\Omega_{p}\left(V\right)$ is such that $S\left(x\right)=S\left(y\right)$,
then $S\left(\mathcal{J}\left(x\right)\right)=S\left(\mathcal{J}\left(y\right)\right)$.
\end{lem}

\begin{lem}
Let $W$ be a Banach space and $\Phi:W\rightarrow\mathbb{R}^{d}$
be a continuous linear functional on $W$. Then there exists a map
$\mathbf{F}:WG\Omega_{p}\left(W\right)\rightarrow WG\Omega_{p}\left(\mathbb{R}^{d}\right)$
such that for all $x\in WG\Omega_{p}\left(W\right)$:

1. $\pi_{1}\left(\mathbf{F}\left(x\right)\right)=\Phi\left(\pi_{1}\left(x\right)\right);$

2. If $x,y\in WG\Omega_{p}\left(W\right)$ is such that $S\left(x\right)=S\left(y\right)$,
then $S\left(\mathbf{F}\left(x\right)\right)=S\left(\mathbf{F}\left(y\right)\right)$. 
\end{lem}

Let $x:\left[0,T\right]\rightarrow G^{\left(\lfloor p\rfloor\right)}$
be a weakly geometric rough path. Define $\overleftarrow{x}:\left[0,T\right]\rightarrow G^{\left(\lfloor p\rfloor\right)}$
by 
\[
\overleftarrow{x_{t}}=x_{T-t}.
\]

\begin{lem}
\label{reversal}Let $x$ be a weakly geometric rough path. Then $S\left(\overleftarrow{x}\right)_{0,T}=S\left(x\right)_{0,T}^{-1}$. 
\end{lem}

\begin{lem}
Let $x\in WG\Omega_{p}\left(V\right)$. Let $\sigma$ be a continuous
non-decreasing function. Then $S\left(x\circ\sigma\right)=S\left(x\right)$. 
\end{lem}

Let $\tilde{T}_{i.r.c.}\left(\left(V\right)\right)$ ($i.r.c.$ for
infinite radius of convergence) be the set of elements in $\tilde{T}\left(\left(V\right)\right)$
such that 
\[
\|\cdot\|:=\max_{k\in\mathbb{N}}\|\pi_{k}\left(\cdot\right)\|_{V^{\otimes k}}^{\frac{1}{k}}<\infty.
\]
By Lemma \ref{symmetric norm}, $\|a\|=\|a^{-1}\|$ and if $a,b\in\tilde{T}_{i.r.c.}\left(\left(V\right)\right)$,
then $a\otimes b\in\tilde{T}_{i.r.c.}\left(\left(V\right)\right)$
and $\left\Vert a+b\right\Vert \leq\left\Vert a\right\Vert +\left\Vert b\right\Vert $.
We define a metric $d$ on $\tilde{T}_{i.r.c.}\left(\left(V\right)\right)$
by 
\[
d\left(a,b\right)=\left\Vert a^{-1}\otimes b\right\Vert .
\]
Subsequently, when there is no confusion, we will use the shorthand
$ab$ to denote $a\otimes b$. 
\begin{lem}
Let $\left[x\left(n\right)\right]_{n=0}^{\infty}$ be a sequence in
$WG\Omega_{p}\left(V\right)$ such that $\sup_{n}\left\Vert x\left(n\right)\right\Vert _{p-var}<\infty$
and $\left[x\left(n\right)\right]_{n=0}^{\infty}$ converges uniformly
to $x$. Then $\left[x\left(n\right)\right]_{n=0}^{\infty}$ has a
subsequence $\left[x\left(n_{k}\right)\right]_{k=0}^{\infty}$ such
that $S\left(x\left(n_{k}\right)\right)\rightarrow S\left(x\right)$.
\end{lem}

\subsection*{Symmetric norm }

The following Lemma useful as finite dimensional projection type of
result and has first appeared in \cite{Cass infinite dimensional rp}. 
\begin{lem}
\label{finite dimensional projection}Let $V$ be a Banach space and
$M\in\mathbb{N}$. Let $X\in\exp\mathcal{L}_{a}^{M}\left(V\right)$.
Then there exists a finite dimensional subspace $V^{\prime}$ of $V$
such that $X\in\exp\mathcal{L}_{a}^{M}\left(V^{\prime}\right)$.\end{lem}
\begin{proof}
Since $X\in\bigoplus_{k=0}^{M}V^{\otimes_{a}k}$, there exists a finite
set $J\subset\mathbb{N}$, $\big(X_{ji}^{k}\big)_{1\leq i\leq k,\; j\in J}$,
$X_{ji}^{k}\in V$ such that 
\[
X=\mathbf{1}+\sum_{k=1}^{M}\sum_{j\in J}X_{j1}^{k}\otimes\ldots\otimes X_{jk}^{k}.
\]
Define $\log:\tilde{T}^{\left(n\right)}\left(V\right)\rightarrow\tilde{T}^{\left(n\right)}\left(V\right)$
by 
\begin{equation}
\log\left(1+g\right)=\pi^{\left(n\right)}\left(\sum_{k=1}^{n}\left(-1\right)^{k+1}\frac{g^{k}}{k}\right).\label{eq:1}
\end{equation}
There exists a finite set $J^{\prime}\subset\mathbb{N}$, $\big(l_{ji}^{k}\big)_{1\leq i\leq k},$$l_{ji}^{k}\in V$,
such that 
\[
\log X=\sum_{k=1}^{M}\sum_{j\in J^{\prime}}\left[l_{j1}^{k},\left[\ldots\left[l_{j\left(k-1\right)}^{k},l_{jk}^{k}\right]\right]\right],
\]
where $\left[c,d\right]=c\otimes d-d\otimes c$. Let 
\[
V^{\prime}:=\mbox{Span}\left\{ l_{ji}^{k}:1\leq i\leq k,\;1\leq k\leq M,j\in J^{\prime}\right\} .
\]
In particular, $X$ lies in $\tilde{T}^{\left(M\right)}\left(V^{\prime}\right)$
and $\log X\in\mathcal{\mathcal{L}}_{a}^{M}\left(V^{\prime}\right)$. 
\end{proof}

\begin{proof}[Proof of Lemma 1.1]As the map $g\rightarrow g^{-1}$
is continuous with respect to the $\left\Vert \cdot\right\Vert _{\tilde{T}^{\left(n\right)}\left(V\right)}$,
it suffices to prove the Lemma for elements in $\exp\mathcal{L}_{a}^{n}\left(V\right)$.
Let$g\in\exp\mathcal{L}_{a}^{n}\left(V\right)$. By Lemma \ref{finite dimensional projection},
there exists a finite dimensional subspace $V^{\prime}$ of $V$ such
that $g\in\exp\mathcal{L}_{a}^{n}\left(V^{\prime}\right)$. Define
the antipode operator $\alpha:G^{\left(*\right)}\rightarrow G^{\left(*\right)}$
(see also p13, \cite{CL14}) so that 
\[
\pi_{k}\left(\alpha\left(v_{1}\otimes\ldots\otimes v_{k}\right)\right)=\left(-1\right)^{k}v_{k}\otimes\ldots\otimes v_{1}.
\]
It follows from Theorem 3.3.3 in \cite{MR2036784} that 
\begin{equation}
\pi_{n}\left(g^{-1}\right)=\pi_{n}\left(\alpha\left(g\right)\right).\label{eq:level wise symmetric norm}
\end{equation}
Therefore, in particular, we have that $\left\Vert \cdot\right\Vert _{\tilde{T}^{\left(n\right)}\left(V\right)}$
is symmetric. \end{proof}

\subsection*{Transformation on rough paths}

\subsubsection*{Signature of signature }

A key property of iterated integrals not just for our purpose but
also for rough path theory in general, is that the iterated integrals
of iterated integrals can be rewritten as a linear combination of
iterated integrals. We will now make this precise in the infinite
dimensional setting. 

Fix $N\in\mathbb{N},$ define 
\[
W=\bigoplus_{i=1}^{N}V^{\otimes i}.
\]
It follows that 
\[
W^{\otimes_{a}n}=\bigoplus_{i_{1},\cdots,i_{n}=1}^{N}V^{\otimes i_{1}}\otimes_{a}\cdots\otimes_{a}V^{\otimes i_{n}}\subset\bigoplus_{i_{1},\cdots,i_{n}=1}^{N}V^{\otimes(i_{1}+\cdots+i_{n})}.
\]
We define the tensor norm $\|\cdot\|_{W^{\otimes n}}$ on $W^{\otimes_{a}n}$
by 
\[
\|\xi\|_{W^{\otimes n}}=\sum_{i_{1},\cdots,i_{n}=1}^{N}\|\xi^{i_{1},\cdots,i_{n}}\|_{V^{\otimes(i_{1}+\cdots+i_{n})}},
\]
where $\xi=(\xi^{i_{1},\cdots,i_{n}})_{1\leq i_{1},\cdots,i_{n}\leq N}$.
Note that by the admissibility of $\left(\left\Vert \cdot\right\Vert _{V^{\otimes k}}:k\geq1\right)$,
\begin{eqnarray*}
 &  & \left(\sum_{1\leq i_{1},\ldots,i_{n}\leq N}\left\Vert \xi^{i_{1},\ldots,i_{n}}\right\Vert _{V^{\otimes\left(i_{1}+\ldots+i_{n}\right)}}\right)\left(\sum_{1\leq j_{1},\ldots,j_{k}\leq N}\left\Vert \xi^{j_{1},\ldots,j_{k}}\right\Vert _{V^{\otimes\left(j_{1}+\ldots+j_{k}\right)}}\right)\\
 & = & \sum_{1\leq i_{1},\ldots,i_{n}\leq N,\;1\leq j_{1},\ldots,j_{k}\leq N}\left\Vert \xi^{i_{1},\ldots,i_{n}}\right\Vert _{V^{\otimes\left(i_{1}+\ldots+i_{n}\right)}}\left\Vert \xi^{j_{1},\ldots,j_{k}}\right\Vert _{V^{\otimes\left(j_{1}+\ldots+j_{k}\right)}}\\
 & \geq & \sum_{1\leq i_{1},\ldots,i_{n+k}\leq N}\left\Vert \xi^{i_{1},\ldots,i_{n+k}}\right\Vert _{V^{\otimes\left(i_{1}+\ldots+i_{n+k}\right)}}.
\end{eqnarray*}
 By taking completion, $W^{\otimes n}$ coincides with $\bigoplus_{i_{1},\cdots,i_{n}=1}^{N}V^{\otimes(i_{1}+\cdots+i_{n})}.$
Let $\overline{\mathcal{L}_{a}^{n}\left(V\right)}$ be the closure
of $\mathcal{L}_{a}^{n}\left(V\right)$ in $\tilde{T}^{\left(n\right)}\left(V\right)$.
By admissibility (\ref{eq:admissible norm}) of the tensor norm $\left\Vert \cdot\right\Vert _{V^{\otimes m}}$,
the multiplication operator $\otimes:\left(V^{\otimes k},V^{\otimes k^{\prime}}\right)\rightarrow V^{\otimes\left(k+k^{\prime}\right)}$
is continuous. This implies that the functions $\exp$ and log are
both continuous. Hence 
\[
\overline{\exp\left(\mathcal{L}_{a}^{n}\left(V\right)\right)}=\exp\left(\overline{\mathcal{L}_{a}^{n}\left(V\right)}\right).
\]
 \begin{proof} [Proof of Lemma 1.2]

Define for $w\in V^{\otimes j}$ and $v\in V^{\otimes k}$, 
\[
F_{j}\left(v\right)\left(w\right)=v\otimes w
\]
and for $v_{1}\in V^{\otimes k_{1}},\ldots,v_{n}\in V^{\otimes k_{n}}$
and $w_{1}\in V^{\otimes j_{1}},\ldots,w_{n}\in V^{\otimes j_{n}}$,
\[
F_{j_{1},\ldots,j_{n}}\left(v_{1},\ldots,v_{n}\right)\left(w_{1}\otimes\ldots\otimes w_{n}\right)=F_{j_{1}}\left(v_{1}\right)\left(w_{1}\right)\otimes\ldots\otimes F_{j_{n}}\left(v_{n}\right)\left(w_{n}\right).
\]
Define for a permutation $\pi$ on $\left\{ 1,\ldots,n\right\} $,
the operator $\pi$ on $V^{\otimes n}$ by 
\[
\pi\left(v_{1}\otimes\ldots\otimes v_{n}\right)=v_{\pi\left(1\right)}\otimes\ldots\otimes v_{\pi\left(n\right)}.
\]
We define for $X,Y\in T^{\left(nN\right)}\left(V\right)$, 
\begin{equation}
H_{i_{1},\ldots,i_{n}}\left(X,Y\right)=\sum_{j_{n}=1}^{i_{n}}\ldots\sum_{j_{1}=1}^{i_{1}}F_{j_{1},\ldots,j_{n}}\left(X^{i_{1}-j_{1}},\ldots,X^{i_{n}-j_{n}}\right)\left[\sum_{\pi\in OS\left(j_{1},\ldots,j_{n}\right)}\pi\left(Y^{j_{1}+\ldots+j_{n}}\right)\right]\label{eq:define H}
\end{equation}
where for $Z\in T^{\left(nN\right)}\left(V\right)$, we use the notation
$Z^{k}$ to denote $\pi_{k}\left(Z\right)$ and $OS\left(j_{1},\ldots,j_{n}\right)$
to denote the set of ordered shuffles (see p72 \cite{lyons st flours}).
We define $\mathcal{J}\left(x\right)$ so that 
\[
S\left(\mathcal{J}\left(x\right)\right)_{s,t}=H_{i_{1},\ldots,i_{n}}\left(S_{nN}\left(x\right)_{0,s},S_{nN}\left(x\right)_{s,t}\right).
\]
As each $j_{1},\ldots,j_{N}$ in the sum in (\ref{eq:define H}) are
at least $1$, $S\left(\mathcal{J}\left(x\right)\right)$ has finite
$p$-variation for $x\in WG\Omega_{p}\left(V\right)$. It remains
to show that $S\left(\mathcal{J}\left(x\right)\right)$ is multiplicative
and takes value in the space of group-like elements. In the finite
dimension case, we have proved these properties of $S\left(\mathcal{J}\left(x\right)\right)$
as lemma 4.4 in \cite{our uniqueness paper }. We will now use Lemma
\ref{finite dimensional projection} to carry out a finite dimensional
approximation. Let $\left(s,u,t\right)\in\left[0,1\right]^{3}$ be
such that $s\leq u\leq t$. Let $1\leq n\leq\lfloor p\rfloor$.By
Corollary 3.9 in \cite{Cass infinite dimensional rp}, $\pi^{\left(nN\right)}\left(S\left(x\right)_{0,s}\right),\;\pi^{\left(nN\right)}\left(S\left(x\right)_{s,u}\right)$
and $\pi^{\left(nN\right)}\left(S\left(x\right)_{u,t}\right)$ all
lie in $G^{\left(nN\right)}$. Therefore, there exist sequences $x^{r},y^{r},z^{r}\in\exp\mathcal{L}_{a}^{nN}\left(V\right)$
such that $x^{r}\rightarrow\pi^{\left(nN\right)}\left(S\left(x\right)_{0,s}\right)$,
$y^{r}\rightarrow\pi^{\left(nN\right)}\left(S\left(x\right)_{s,u}\right)$
and $z^{r}\rightarrow\pi^{\left(nN\right)}\left(S\left(x\right)_{u,t}\right)$
as $r\rightarrow\infty$. By Lemma \ref{finite dimensional projection},
there exists finite dimensional spaces $V_{r}$ such that $x_{r},y_{r},z_{r}\in\exp\mathcal{L}_{a}^{nN}\left(V_{r}\right)$.
By the Chow-Rashevskii's theorem (\cite{Rashevskii theorem,Chow39}
or see for example Theorem 7.28 in \cite{FV10}), there exist bounded
variation paths $a$,$b$,$c$ in $V_{r}$ such that $S_{nN}\left(a\right)_{0,s}=x^{r}$,
$S_{nN}\left(b\right)_{s,u}=y^{r}$ and $S_{nN}\left(c\right)_{u,t}=z^{r}$.
Let $\xi$ denote the path $a\star b\star c$, where $\star$ denotes
the concatenation of paths. Then for $i_{1}+\ldots+i_{n}\leq nN$,
\begin{eqnarray*}
H_{i_{1},\ldots,i_{n}}\left(x^{r},y^{r}z^{r}\right) & = & H_{i_{1},\ldots,i_{n}}\left(S_{nN}\left(\xi\right)_{0,s},S_{nN}\left(\xi\right)_{s,t}\right)
\end{eqnarray*}
and the same holds when $s$, $x^{r}$ and $y^{r}z^{r}$ are replaced,
respectively by $u$, $x^{r}y^{r}$ and $z^{r}$. 

By the computation leading to (4.4) in \cite{our uniqueness paper },
for $i_{1}+\ldots+i_{n}\leq nN$, 
\begin{eqnarray*}
 &  & \int_{s<s_{1}<\ldots<s_{n}<t}\mathrm{d}S_{nN}\left(\xi\right)_{0,s_{1}}^{i_{1}}\otimes\ldots\otimes\mathrm{d}S_{nN}\left(\xi\right)_{0,s_{n}}^{i_{n}}\\
 & = & H_{i_{1},\ldots,i_{n}}\left(S_{nN}\left(\xi\right)_{0,s},S_{nN}\left(\xi\right)_{s,t}\right).
\end{eqnarray*}
The same holds when $\left(s,u\right)$ is replaced by $\left(u,t\right)$.
As the iterated integrals of the bounded variation path 
\[
s\rightarrow S_{nN}\left(\xi\right)_{0,s}
\]
has the multiplicative property, we also have 

\begin{eqnarray*}
H_{i_{1},\ldots,i_{n}}\left(x^{r},y^{r}z^{r}\right) & = & \sum_{j=0}^{n}H_{i_{1},\ldots,i_{j}}\left(x^{r},y^{r}\right)\otimes H_{i_{j+1},\ldots,i_{n}}\left(x^{r}y^{r},z^{r}\right).
\end{eqnarray*}
Moreover, as the signature of a bounded variation path $\xi$ is a
group-like element, for all $t$, 
\begin{eqnarray*}
\sum_{i_{n}=0}^{N}\ldots\sum_{i_{1}=0}^{N}H_{i_{1},\ldots,i_{n}}\left(\mathbf{1},S_{nN}\left(\xi\right)_{0,t}\right) & \in & \exp\left(\mathcal{L}_{a}^{nN}\left(V_{r}\right)\right)\subset\exp\left(\mathcal{L}_{a}^{nN}\left(V\right)\right).
\end{eqnarray*}
By taking limit as $r\rightarrow\infty$ and using the continuity
of the map $H$, the map $H$ satisfies the multiplicative property
\begin{eqnarray*}
 &  & \sum_{j=0}^{n}H_{i_{1},\ldots,i_{j}}\left(S_{nN}\left(x\right)_{0,s},S_{nN}\left(x\right)_{s,u}\right)\otimes H_{i_{j+1},\ldots,i_{n}}\left(S_{nN}\left(x\right)_{0,u},S_{nN}\left(x\right)_{u,t}\right)\\
 & = & H_{i_{1},\ldots,i_{n}}\left(S_{nN}\left(x\right)_{0,s},S_{nN}\left(x\right)_{s,t}\right).
\end{eqnarray*}
and lies in the space of group-like elements 
\[
\sum_{i_{n}=0}^{N}\ldots\sum_{i_{1}=0}^{N}H_{i_{1},\ldots,i_{n}}\left(\mathbf{1},S_{nN}\left(x\right)_{0,t}\right)\in\overline{\exp\left(\mathcal{L}_{a}^{nN}\left(V\right)\right)}.
\]
This holds for all $s,u$ and $t$ and hence $Z$ is a weakly geometric
rough path. 

\end{proof}

\subsubsection*{Linear map on rough paths }

\begin{proof}[Proof of Lemma 1.3]

For any linear map $\Phi:W\rightarrow\mathbb{R}^{d}$, we may continuously
extend $\mathbf{\Phi}$ to a linear operator on $T^{\left(N\right)}\left(W\right)$
such that for $w_{1},\ldots,w_{N}\in W$, 
\[
\mathbf{\Phi}\left(w_{1}\otimes\ldots\otimes w_{N}\right)=\mathbf{\Phi}\left(w_{1}\right)\otimes\ldots\otimes\mathbf{\Phi}\left(w_{N}\right).
\]
Let $x\in WG\Omega_{p}\left(W\right)$. As $\Phi$ is a bounded linear
operator and the norm on $T^{\left(N\right)}\left(W\right)$ is admissible,
$\Phi\left(x\right)$ has finite $p$-variation. As $\Phi$ is a homomorphism
with respect to $\otimes$, for all $t\geq0$, $\Phi\left(x_{t}\right)$
lies in the $\lfloor p\rfloor$-step free nilpotent Lie group over
$\mathbb{R}^{d}$. Therefore, $\Phi\left(x\right)\in WG\Omega_{p}\left(\mathbb{R}^{d}\right)$.
By construction, $\pi_{1}\left(\Phi\left(x\right)\right)=\Phi\left(\pi_{1}\left(x\right)\right)$.
Moreover, again by the homomorphism property of $\Phi$ and admissibility
of the norm on $T^{\left(N\right)}\left(W\right)$, we have 
\[
\Phi\left(S_{N}\left(x\right)\right)=S_{N}\left(\Phi\left(x\right)\right)
\]
which implies property 2. in the Lemma. 

\end{proof}

\subsection*{Signature of the reversed path }

Let $x$ be a weakly geometric rough path on an interval $\left[0,T\right]$.
We recall that the reversal of $x$, denoted as $\overleftarrow{x}$,
is defined so that 
\[
\overleftarrow{x_{t}}=x_{T-t}.
\]
We will now prove Lemma \ref{reversal} that 
\[
S\left(x\right)_{0,T}\otimes S\left(\overleftarrow{x}\right)_{0,T}=\mathbf{1}.
\]

\begin{proof}[Proof of Lemma 1.4]

We will prove a stronger fact that for all $s\leq t$ and all $n$,
\[
S_{n}\left(x\right)_{s,t}\otimes S_{n}\left(\overleftarrow{x}\right)_{T-t,T-s}=\mathbf{1}
\]
by induction on $n$, where $S_{n}\left(x\right)=\pi^{\left(n\right)}\left(S\left(x\right)\right)$
and $S\left(x\right)_{s,t}=S\left(x|_{\left[s,t\right]}\right)$.
The base induction case of $n=\lfloor p\rfloor$ is obvious since
$S_{\lfloor p\rfloor}\left(\overleftarrow{x}\right)_{T-t,T-s}$ is
by definition equals to $\overleftarrow{x}\,_{T-t}^{-1}\overleftarrow{x}\,_{T-s}=x_{t}^{-1}x_{s}$.
We will use the notation $\mathbb{X}^{i}$ and $\overleftarrow{\mathbb{X}^{i}}$
to denote, respectively, $\pi_{i}\left(S\left(x\right)\right)$ and
$\pi_{i}\left(S\left(\overleftarrow{x}\right)\right)$. By (2.2.9)
in \cite{MR1654527}, 
\begin{eqnarray*}
 &  & S_{n+1}\left(x\right)_{s,t}\otimes S_{n+1}\left(\overleftarrow{x}\right)_{T-t,T-s}\\
 & = & \lim_{\max_{i}\left|u_{i+1}-u_{i}\right|\rightarrow0}\hat{\mathbb{X}}_{u_{0},u_{1}}\otimes\ldots\otimes\hat{\mathbb{X}}_{u_{l-1},u_{l}}\otimes\lim_{\max\left|s_{i+1}-s_{i}\right|\rightarrow0}\hat{\overleftarrow{\mathbb{X}}}_{s_{0},s_{1}}\otimes\ldots\otimes\hat{\overleftarrow{\mathbb{X}}}_{s_{k-1},s_{k}}.
\end{eqnarray*}
where $\hat{\mathbb{X}}_{s,t}$ and $\hat{\overleftarrow{\mathbb{X}}}_{s,t}$
denote respectively elements $\left(S_{n}\left(x\right)_{s,t},0\right)$
and $\left(S_{n}\left(\overleftarrow{x}\right)_{s,t},0\right)$ in
$T^{\left(n+1\right)}\left(V\right)$, $s=u_{0}<u_{1}<\ldots<u_{l}=t$
and $T-t=s_{0}<s_{1}<\ldots<s_{k}=T-s$. In particular, we have 
\begin{eqnarray*}
 &  & S_{n+1}\left(x\right)_{s,t}\otimes S_{n+1}\left(\overleftarrow{x}\right)_{T-t,T-s}\\
 & = & \lim_{\max\left|u_{i+1}-u_{i}\right|\rightarrow0}\hat{\mathbb{X}}_{u_{0},u_{1}}\otimes\ldots\otimes\hat{\mathbb{X}}_{u_{l-1},u_{l}}\otimes\hat{\overleftarrow{\mathbb{X}}}_{T-u_{l},T-u_{l-1}}\otimes\ldots\otimes\hat{\overleftarrow{\mathbb{X}}}_{T-u_{1},T-u_{0}}
\end{eqnarray*}
where the tensor product $\otimes$ is taken in $T^{\left(n+1\right)}\left(V\right)$.
We claim that for all $t_{0}<\ldots<t_{l}$, 
\begin{eqnarray*}
 &  & \hat{\mathbb{X}}_{t_{0},t_{1}}\otimes\ldots\otimes\hat{\mathbb{X}}_{t_{l-1},t_{l}}\otimes\hat{\overleftarrow{\mathbb{X}}}_{T-t_{l},T-t_{l-1}}\otimes\ldots\otimes\hat{\overleftarrow{\mathbb{X}}}_{T-t_{1},T-t_{0}}\\
 & = & \left(1,0\ldots,0,\sum_{j=0}^{l-1}\sum_{k=1}^{n}\mathbb{X}_{t_{j},t_{j+1}}^{k}\otimes\overleftarrow{\mathbb{X}}\,_{T-t_{j+1},T-t_{j}}^{n+1-k}\right)
\end{eqnarray*}
by induction on $l$. For the base induction step, note that by the
induction hypothesis over $n$, 
\begin{eqnarray*}
\pi^{\left(n\right)}\left(\hat{\mathbb{X}}_{t_{0},t_{1}}\otimes\hat{\overleftarrow{\mathbb{X}}}_{T-t_{1},T-t_{0}}\right) & = & S_{n}\left(x\right)_{t_{0},t_{1}}\otimes S_{n}\left(\overleftarrow{x}\right)_{T-t_{1},T-t_{0}}\\
 & = & \left(1,0\ldots,0\right)
\end{eqnarray*}
and 
\[
\pi_{n+1}\left(\hat{\mathbb{X}}_{t_{0},t_{1}}\otimes\hat{\overleftarrow{\mathbb{X}}}_{T-t_{1},T-t_{0}}\right)=\sum_{k=1}^{n}\mathbb{X}_{t_{0},t_{1}}^{k}\otimes\overleftarrow{\mathbb{X}}\,_{T-t_{1},T-t_{0}}^{n+1-k}.
\]
For the induction step, we have by induction hypothesis that 
\begin{eqnarray*}
 &  & \hat{\mathbb{X}}_{t_{1},t_{2}}\otimes\ldots\otimes\hat{\mathbb{X}}_{t_{l-1},t_{l}}\otimes\hat{\overleftarrow{\mathbb{X}}}_{T-t_{l},T-t_{l-1}}\otimes\ldots\otimes\hat{\overleftarrow{\mathbb{X}}}_{T-t_{2},T-t_{1}}\\
 & = & \left(1,0\ldots,0,\sum_{j=1}^{l-1}\sum_{k=1}^{n}\mathbb{X}_{t_{j},t_{j+1}}^{k}\otimes\overleftarrow{\mathbb{X}}\,_{T-t_{j+1},T-t_{j}}^{n+1-k}\right).
\end{eqnarray*}
To conclude the proof for the claim, it suffices to observe that 
\begin{eqnarray*}
 &  & \left(S_{n}\left(x\right)_{t_{0},t_{1}},0\right)\otimes\left(1,0\ldots,0,\sum_{j=1}^{l-1}\sum_{k=1}^{n}\mathbb{X}_{t_{j},t_{j+1}}^{k}\otimes\overleftarrow{\mathbb{X}}\,_{T-t_{j+1},T-t_{j}}^{n+1-k}\right)\otimes\left(S_{n}\left(\overleftarrow{x}\right)_{T-t_{1},T-t_{0}},0\right)\\
 & = & \left(S_{n}\left(x\right)_{t_{0},t_{1}},\sum_{j=1}^{l-1}\sum_{k=1}^{n}\mathbb{X}_{t_{j},t_{j+1}}^{k}\otimes\overleftarrow{\mathbb{X}}\,_{T-t_{j+1},T-t_{j}}^{n+1-k}\right)\otimes\left(S_{n}\left(\overleftarrow{x}\right)_{T-t_{1},T-t_{0}},0\right)\\
 & = & \left(1,0\ldots,0,\sum_{j=0}^{l-1}\sum_{k=1}^{n}\mathbb{X}_{t_{j},t_{j+1}}^{k}\otimes\overleftarrow{\mathbb{X}}\,_{T-t_{j+1},T-t_{j}}^{n+1-k}\right).
\end{eqnarray*}
We now observe, using the claim, that 
\begin{eqnarray}
 &  & \left\Vert \hat{\mathbb{X}}_{t_{0},t_{1}}\otimes\ldots\otimes\hat{\mathbb{X}}_{t_{l-1},t_{l}}\otimes\hat{\overleftarrow{\mathbb{X}}}_{T-t_{l},T-t_{l-1}}\otimes\ldots\otimes\hat{\overleftarrow{\mathbb{X}}}_{T-t_{1},T-t_{0}}\right\Vert \nonumber \\
 & \leq & \sum_{j=0}^{l-1}\sum_{k=1}^{n}\left\Vert \mathbb{X}_{t_{j},t_{j+1}}^{k}\otimes\overleftarrow{\mathbb{X}}\,_{T-t_{j+1},T-t_{j}}^{n+1-k}\right\Vert \nonumber \\
 & \leq & \sum_{j=0}^{l-1}\sum_{k=1}^{n}\left\Vert \mathbb{X}_{t_{j},t_{j+1}}^{k}\right\Vert \left\Vert \overleftarrow{\mathbb{X}}\,_{T-t_{j+1},T-t_{j}}^{n+1-k}\right\Vert .\label{eq:inital calculation}
\end{eqnarray}
Note that $\overleftarrow{\mathbb{X}}\,_{T-t_{j+1},T-t_{j}}^{n+1-k}=\pi_{n+1-k}\left(S_{n+1-k}\left(x\right)_{t_{j},t_{j+1}}^{-1}\right)$
by induction hypothesis. By (\ref{eq:level wise symmetric norm})
in the proof of Lemma \ref{symmetric norm}, 
\[
\left\Vert \overleftarrow{\mathbb{X}}\,_{T-t_{j+1},T-t_{j}}^{n+1-k}\right\Vert =\left\Vert \mathbb{X}_{t_{j},t_{j+1}}^{n+1-k}\right\Vert .
\]
Therefore, by (\ref{eq:inital calculation}), 
\begin{eqnarray*}
 &  & \left\Vert \hat{\mathbb{X}}_{t_{0},t_{1}}\otimes\ldots\otimes\hat{\mathbb{X}}_{t_{l-1},t_{l}}\otimes\hat{\overleftarrow{\mathbb{X}}}_{T-t_{l},T-t_{l-1}}\otimes\ldots\otimes\hat{\overleftarrow{\mathbb{X}}}_{T-t_{1},T-t_{0}}\right\Vert \\
 & \leq & \sum_{j=0}^{l-1}\sum_{k=1}^{n}\left\Vert x|_{\left[t_{j},t_{j+1}\right]}\right\Vert _{p-var}^{k}\left\Vert x|_{\left[t_{j},t_{j+1}\right]}\right\Vert _{p-var}^{n+1-k}\\
 & \leq & n\max_{j}\left\Vert x|_{\left[t_{j},t_{j+1}\right]}\right\Vert _{p-var}^{n+1-p}\left\Vert x|_{\left[s,t\right]}\right\Vert _{p-var}^{p}\\
 & \rightarrow & 0
\end{eqnarray*}
as $\max_{j}\left|t_{j+1}-t_{j}\right|\rightarrow0$. 

\end{proof}

\subsection*{Invariance of signature under reparametrisation }

\begin{proof}[Proof of Lemma 1.5]

We shall prove it by induction on $n$ that 
\[
S_{n}\left(x\circ\sigma\right)_{0,T}=S_{n}\left(x\right)_{0,T}.
\]
We will prove the stronger fact that $S_{n}\left(x\circ\sigma\right)_{s,t}=S_{n}\left(x\right)_{\sigma\left(s\right),\sigma\left(t\right)}$
for all $s,t\in\left[0,T\right]$. Let $S_{n}\left(x\circ\sigma\right)_{s,t}=\left(1,\mathbb{Y}_{s,t}^{1},\ldots,\mathbb{Y}_{s,t}^{n}\right)$.

Note that 
\begin{eqnarray*}
\mathbb{Y}_{s,t}^{n+1} & = & \lim_{\left|\mathcal{P}\right|\rightarrow0}\sum_{j=0}^{l-1}\sum_{k=1}^{n}\mathbb{Y}_{s,t_{j}}^{k}\otimes\mathbb{Y}_{t_{j},t_{j+1}}^{n+1-k}\\
 & = & \lim_{\left|\mathcal{P}\right|\rightarrow0}\sum_{j=0}^{l-1}\sum_{k=1}^{n}\mathbb{X}_{\sigma\left(s\right),\sigma\left(t_{j}\right)}^{k}\otimes\mathbb{X}_{\sigma\left(t_{j}\right),\sigma\left(t_{j+1}\right)}^{n+1-k},
\end{eqnarray*}
where the second line follows from the induction hypothesis. Let $\left(s_{j}\right)$
be a sequence defined by $s_{0}=s$ and 
\[
s_{j+1}=\min_{k}\left\{ t_{k}:\sigma\left(t_{k}\right)>\sigma\left(s_{j}\right)\right\} .
\]
Then as by definition $\mathbb{X}_{u,u}^{m}=0$ for any $m$ and any
$u$, 
\begin{equation}
\mathbb{Y}_{s,t}^{n+1}=\lim_{\left|\mathcal{P}\right|\rightarrow0}\sum_{j}\sum_{k=1}^{n}\mathbb{X}_{\sigma\left(s\right),\sigma\left(s_{j}\right)}^{k}\otimes\mathbb{X}_{\sigma\left(s_{j}\right),\sigma\left(s_{j+1}\right)}^{n+1-k}.\label{eq:integral}
\end{equation}
Note that as $\sigma$ is non-decreasing, the set $\left(\sigma\left(s_{j}\right):j\geq0\right)$
is a partition of $\left[\sigma\left(s\right),\sigma\left(t\right)\right]$.
Moreover, since for all $j$, there exists $m$ such that $\left(s_{j},s_{j+1}\right)=\left(\sigma\left(t_{m}\right),\sigma\left(t_{m+1}\right)\right)$
and $\sigma$ is uniformly continuous, we have that $\max_{j}\left|s_{j+1}-s_{j}\right|\rightarrow0$
as $\max_{m}\left|t_{m+1}-t_{m}\right|\rightarrow0$. Therefore, by
(\ref{eq:integral}), $\mathbb{Y}_{s,t}^{n+1}=\mathbb{X}_{\sigma\left(s\right),\sigma\left(t\right)}^{n+1}$. 

\end{proof}

\subsection{Uniform convergence $+$ Bounded in $p$-variation $\implies$ Subsequential
convergence of Signatures}

The following proof is simply a matter of checking that the proof
in \cite{FV10} applies in the infinite dimensional setting.

\begin{proof}[Proof of Lemma 1.6]

By Theorem 2.2.2 in \cite{MR1654527}, it suffices to show that there
exists $p^{\prime}$ and a control $\omega$ (see Definition 1.9 in
\cite{lyons st flours}) such that for all $1\leq i\leq\lfloor p\rfloor$
and $s\leq t$, 
\[
\left\Vert \pi_{i}\left(x\left(n_{k}\right)_{s,t}\right)\right\Vert \leq\omega\left(s,t\right)^{\frac{i}{p^{\prime}}},\;\left\Vert \pi_{i}\left(x_{s,t}\right)\right\Vert \leq\omega\left(s,t\right)^{\frac{i}{p^{\prime}}}
\]
and a sequence $a\left(n_{k}\right)\rightarrow0$ such that 
\[
\left\Vert \pi_{i}\left(x\left(n_{k}\right)_{s,t}-x_{s,t}\right)\right\Vert \leq a\left(n_{k}\right)\omega\left(s,t\right)^{\frac{i}{p^{\prime}}}.
\]
By Theorem 3.3.3 in \cite{MR1654527}, it suffices to show that for
some $p^{\prime}\geq1$, 
\[
d_{p^{\prime}-var}\left(x\left(n\right),x\right)\rightarrow0
\]
as $n\rightarrow\infty$, where $d_{p^{\prime}-var}$ is defined by
(\ref{eq:p-var}) below.

Take $p^{\prime}>p$ such that $\lfloor p\rfloor=\lfloor p^{\prime}\rfloor$.
For $x\in WG\Omega_{p}\left(V\right)$, let $x_{s,t}$ denote $x_{s}^{-1}x_{t}$.
Then 
\begin{eqnarray}
d_{p^{\prime}-var}\left(x,y\right) & = & \max_{1\leq i\leq\lfloor p^{\prime}\rfloor}\left(\sum_{j}\left\Vert \pi_{i}\left(x_{t_{j},t_{j+1}}-y_{t_{j},t_{j+1}}\right)\right\Vert ^{\frac{p^{\prime}}{i}}\right)^{\frac{i}{p^{\prime}}}\label{eq:p-var}\\
 & \leq & \sup_{s\leq t}\left\Vert x_{s,t}-y_{s,t}\right\Vert ^{\frac{p^{\prime}-p}{p^{\prime}}}\max_{1\leq i\leq\lfloor p^{\prime}\rfloor}\left(\sum_{j}\left\Vert \pi_{i}\left(x_{t_{j},t_{j+1}}-y_{t_{j},t_{j+1}}\right)\right\Vert ^{\frac{p}{i}}\right)^{\frac{i}{p^{\prime}}}\nonumber \\
 & \leq & \sup_{s\leq t}\left\Vert x_{s,t}-y_{s,t}\right\Vert ^{\frac{p^{\prime}-p}{p^{\prime}}}2^{p-1}\max_{1\leq i\leq\lfloor p^{\prime}\rfloor}\left(\sum_{j}\left\Vert \pi_{i}\left(x_{t_{j},t_{j+1}}\right)\right\Vert ^{\frac{p}{i}}+\left\Vert \pi_{i}\left(y_{t_{j},t_{j+1}}\right)\right\Vert ^{\frac{p}{i}}\right)^{\frac{i}{p^{\prime}}}\nonumber \\
 & \leq & \sup_{s\leq t}\left\Vert x_{s,t}-y_{s,t}\right\Vert ^{\frac{p^{\prime}-p}{p^{\prime}}}2^{p-1}\big[\max_{1\leq i\leq\lfloor p^{\prime}\rfloor}\left(\sum_{j}\left\Vert \pi_{i}\left(x_{t_{j},t_{j+1}}\right)\right\Vert ^{\frac{p}{i}}\right)^{\frac{i}{p^{\prime}}}\label{eq:p variation bound}\\
 &  & +\max_{1\leq i\leq\lfloor p^{\prime}\rfloor}\left(\sum_{j}\left\Vert \pi_{i}\left(y_{t_{j},t_{j+1}}\right)\right\Vert ^{\frac{p}{i}}\right)^{\frac{i}{p^{\prime}}}\big].
\end{eqnarray}
Note that 
\begin{eqnarray}
\left\Vert \pi_{i}\left(x_{s,t}-y_{s,t}\right)\right\Vert  & \leq & \left\Vert \pi_{i}\left(\left(x_{s}^{-1}-y_{s}^{-1}\right)x_{t}+y_{s}^{-1}\left(x_{t}-y_{t}\right)\right)\right\Vert \nonumber \\
 & \le & \sum_{j=0}^{i}\left\Vert x_{s}^{-1}-y_{s}^{-1}\right\Vert ^{j}\left\Vert x_{t}\right\Vert ^{i-j}+\left\Vert x_{s}-y_{s}\right\Vert ^{j}\left\Vert y_{s}^{-1}\right\Vert ^{i-j}.\label{eq:uniform bound}
\end{eqnarray}
If $y=x\left(n\right)$ and $x\left(n\right)\rightarrow x$ uniformly,
then since we have for $g=\left(1,g^{1},\ldots,g^{\lfloor p\rfloor}\right)$,
\begin{eqnarray*}
g^{-1} & = & \sum_{j=0}^{\lfloor p\rfloor}\left(-1\right)^{j}\left(g-1\right)^{\otimes j}\\
 & = & \sum_{j=0}^{\lfloor p\rfloor}\left(-1\right)^{j}\sum_{i_{1},\ldots,i_{j}\geq1}g^{i_{1}}\ldots g^{i_{j}}
\end{eqnarray*}
and the tensor product is continuous, 
\[
\lim_{n\rightarrow\infty}\sup_{s}\left\Vert x_{s}^{-1}-x\left(n\right)_{s}^{-1}\right\Vert =0.
\]
By (\ref{eq:uniform bound}), $\sup_{s\leq t}\left\Vert x_{s,t}-x\left(n\right)_{s,t}\right\Vert \rightarrow0$
as $n\rightarrow\infty$ and by (\ref{eq:p variation bound}), $d_{p^{\prime}-var}\left(x\left(n\right),x\right)\rightarrow0$
as $n\rightarrow\infty$. 

\end{proof}

\section*{Partially ordered sets as $\mathbb{R}$-trees}

Finally, we prove the following lemma which characterise partially
ordered sets that can be realised as $\mathbb{R}$-trees. 
\begin{lem}
\label{tree lemma} Let $\left(\mathcal{P},\preceq\right)$ be a partially
ordered set such that: 

1. $\mathcal{P}$ has a least element $v$. 

2. For all $\tau_{1}\in\mathcal{P}$, the set $\left\{ \tau_{2}\in\mathcal{P}:\tau_{2}\preceq\tau_{1}\right\} $
is totally ordered. 

3. For any $\tau_{1},\tau_{2}\in\mathcal{P}$, the set $\left\{ \tau_{3}\in\mathcal{P}:\tau_{3}\preceq\tau_{1},\tau_{3}\preceq\tau_{2}\right\} $
has a unique maximal element, which will be denoted as $\tau_{1}\wedge\tau_{2}$.

4. There exists a function $\alpha:\mathcal{P}\rightarrow\left[0,\infty\right)$
such that $\alpha\left(v\right)=0$ and the restriction of $\alpha$
on the set $\left\{ \tau_{3}\in\mathcal{P}:\tau_{3}\preceq\tau_{1}\right\} $
is strictly increasing for any $\tau_{1}\in\mathcal{P}$. 

then

(i) for all $\tau_{1},\tau_{2},\tau_{3}\in\mathcal{P}$, 
\begin{equation}
\alpha\left(\tau_{1}\wedge\tau_{2}\right)\geq\min\left(\alpha\left(\tau_{1}\wedge\tau_{3}\right),\alpha\left(\tau_{2}\wedge\tau_{3}\right)\right);\label{eq:zero hyperbolic}
\end{equation}

(ii)if we define 
\[
d\left(\tau_{1},\tau_{2}\right)=\alpha\left(\tau_{1}\right)+\alpha\left(\tau_{2}\right)-2\alpha\left(\tau_{1}\wedge\tau_{2}\right),
\]
then $\left(\mathcal{P},d\right)$ is a metric space that is $0$-hyperbolic
with respect to $v$ in the sense of \cite{Chi01} (p11.). 

(iii) $\left(\mathcal{P},d\right)$ is a $\mathbb{R}$-tree.\end{lem}
\begin{rem}
This Lemma is essentially Proposition 3.10 in \cite{Tree partial order }.
Here we provide an alternative proof. \end{rem}
\begin{proof}
(i) As $\left\{ \tau_{4}:\tau_{4}\preceq\tau_{3}\right\} $ is totally
ordered, we may assume without loss of generality that 
\[
\tau_{1}\wedge\tau_{3}\preceq\tau_{2}\wedge\tau_{3}.
\]
This implies in particular that $\tau_{1}\wedge\tau_{3}\preceq\tau_{2}$.
However, by the definition of $\wedge$, we also have $\tau_{1}\wedge\tau_{3}\preceq\tau_{1}$.
Therefore, $\tau_{1}\wedge\tau_{3}\preceq\tau_{1}\wedge\tau_{2}$.
The inequality (\ref{eq:zero hyperbolic}) now follows from the assumption
4. in the Lemma.

(ii) The only thing that needs proving is the triangle inequality
for the metric $d$ and the property of $0$-hyperbolic. Both of which
are equivalent to (\ref{eq:zero hyperbolic}). 

(iii) By Lemma 4.13 in \cite{Chi01}, a metric space is an $\mathbb{R}$-tree
if and only if it is connected and $0$-hyperbolic. The metric space
$\left(\mathcal{P},d\right)$ is path-connected by Condition 4., and
is $0$-hyperbolic by (ii), and hence $\left(\mathcal{P},d\right)$
is a $\mathbb{R}$-tree. 
\end{proof}
\section*{Acknowledgement}
All four authors gratefully acknowledge the support of ERC (Grant Agreement No.291244 Esig). The third author is also supported by EPSRC (EP/F029578/1).

\end{document}